\documentclass[11pt]{amsart}

\usepackage{geometry}
\usepackage[all]{xy}
\usepackage{epstopdf}
\usepackage{amssymb,amsmath,amsfonts,amsthm,graphicx,enumerate,amscd}
\usepackage{mathrsfs}

\newcommand{\deff}[1]{\textbf{\emph{\sharp1}}}

\newcommand{\func}[3]{\sharp1 \colon \sharp2 \to \sharp3}
\newcommand{\bb}[1]{\mathbb{\sharp1}}
\newcommand{\lie}[1]{\mathfrak{\sharp1}}
\newcommand{\iprod}[2]{\langle \sharp1, \sharp2 \rangle}
\newcommand{\ddell}[1]{\frac{\partial}{\partial \sharp1}}

\numberwithin{equation}{section}

\theoremstyle{plain}
\newtheorem{theorem}{Theorem}[section]

\newtheorem{claim}[theorem]{Claim}

\newtheorem{lemma}[theorem]{Lemma}
\newtheorem{proposition}[theorem]{Proposition}

\theoremstyle{definition}

\newtheorem{remark}[theorem]{Remark}

\newtheorem{definition}[theorem]{Definition}

\renewcommand	{\k}		{k}

\title{Equivariant formality in $K$-theory}
\author{Chi-Kwong Fok}
\date{October 3, 2018}

\allowdisplaybreaks

\begin{document}
\maketitle
\begin{abstract}
	In this note we present an analogue of equivariant formality in $K$-theory and show that it is equivalent to equivariant formality \emph{\`a la} Goresky-Kottwitz-MacPherson. We also apply this analogue to give alternative proofs of equivariant formality of conjugation action on compact Lie groups, left translation action on generalized flag manifolds, and compact Lie group actions with maximal rank isotropy subgroups. 
\end{abstract}


\section{Introduction} Equivariant formality, first defined in \cite{GKM}, is a special property of group actions on topological spaces which allows for easy computation of their equivariant cohomology. 
%Roughly speaking, equivariant formality amounts to the existence of equivariant extension in the equivariant cohomology theory of any element in the ordinary cohomology theory. Equivalently, 
A $G$-action on a space $X$ is said to be equivariantly formal if the Leray-Serre spectral sequence for the rational cohomology of the fiber bundle $X\hookrightarrow X\times_G EG\to BG$ collapses on the $E_2$-page. The latter is also equivalent to $H_G^*(X; \mathbb{Q})\cong H_G^*(\text{pt}; \mathbb{Q})\otimes H^*(X; \mathbb{Q})$ as $H_G^*(\text{pt}; \mathbb{Q})$-modules. There are various examples of interest which are known to be equivariantly formal, e.g. Hamiltonian group actions on compact symplectic manifolds and linear algebraic torus actions on smooth complex projective varieties (cf. \cite[Section 1.2 and Theorem 14.1]{GKM}).

Though equivariant formality was first defined in terms of equivariant cohomology, in some situations working with analogous notions phrased in terms of other equivariant cohomology theories may come in handy. The notion of equivariant formality in $K$-theory was introduced and explored by Harada and Landweber in \cite{HL}, where they instead used the term `weak equivariant formality' and exploited this notion to show equivariant formality of Hamiltonian actions on compact symplectic manifolds. 
\begin{definition}[{cf. \cite[Def. 4.1]{HL}}]
\label{def:weakKEF} 
Let $\k$ be a commutative ring, $G$ a compact Lie group and $X$ a $G$-space. We use $K^*(X)$ (resp. $K_G^*(X)$) to denote the $\mathbb{Z}_2$-graded\footnote{Recall that one can use $\mathbb{Z}_2$-grading in defining complex $K$-theory thanks to Bott periodicity.} complex (equivariant) $K$-theory of $X$, and ${K^*(X; k)}$ (resp. ${K_G^*(X; k)}$) to denote $K^*(X)\otimes k$ (resp. $K^*_G(X)\otimes k$). We denote the complex representation ring of $G$ by $R(G)$, and write ${R(G;\k)}:= R(G) \otimes \k$, and ${I(G;\k)}=I(G)\otimes k$, where $I(G)$ is the augmentation ideal of $R(G)$. Let 
\[f_G: K^*_G(X)\to K^*(X)\]
be the forgetful map. A $G$-action on a space $X$ is $k$-weakly equivariantly formal if $f_G$ induces an isomorphism
\[
	K_G^*(X; k)\otimes_{R(G; k)}k\to K^*(X; k)
\]
We simply say the action is weakly equivariantly formal
in the case $\k = \mathbb{Z}$.
\end{definition}

Harada and Landweber settled for weakly equivariant formality as in Definition \ref{def:weakKEF} as the $K$-theoretic analogue of equivariant formality, instead of the seemingly obvious candidate $K_G^*(X)\cong K_G^*(\text{pt})\otimes K^*(X)$, citing the lack of the Leray-Serre spectral sequence for Atiyah-Segal's equivariant $K$-theory. The term `weak' is in reference to the condition in Definition \ref{def:weakKEF} being weaker than $K_G^*(X)\cong K_G^*(\text{pt})\otimes K^*(X)$ because of the possible presence of torsion. We would like to define the following version of $K$-theoretic equivariant formality in exact analogy with another cohomological equivariant formality condition that the forgetful map $H_G^*(X)\to H^*(X)$ be onto. 
\begin{definition}\label{rkef}
We say that $X$ is a \emph{rational $K$-theoretic equivariantly formal} (\emph{RKEF} for short) $G$-space if the forgetful map
\[f_G\otimes \text{Id}_\mathbb{Q}: K_G^*(X; \mathbb{Q})\to K^*(X; \mathbb{Q})\]
is onto.
\end{definition}
Recall that $K^0(X)$ (resp. $K^{-1}(X)$) is the Grothendieck group of the commutative monoid of isomorphism classes of (resp. reduced) complex vector bundles over $X$ (resp. $\Sigma X$) under Whitney sum, and $K_G^*(X)$ can be similarly defined using equivariant vector bundles. The above condition then admits a natural interpretation in terms of vector bundles: for every vector bundle $V$ over $X$ and its suspension $\Sigma X$, there are natural numbers $p, q$ such that $V^{\oplus p} \oplus \underline{\mathbb{C}}^q$ admits an equivariant $G$-structure. 

In this note, we will prove the following theorem, which asserts the equivalence of RKEF and equivariant formality in the classical sense. %The main result Theorem \ref{rkef&rcef} in this Section comes from work in preparation (\cite{F2}). For convenience of the reader we reproduce its proof here. 

\begin{theorem}\label{rkef&rcef} Let $G$ be a compact and connected Lie group which acts on a finite CW-complex $X$. The following are equivalent. 
\begin{enumerate}
	\item $X$ is a RKEF $G$-space. 
	\item $X$ is an equivariantly formal $G$-space.
	\item $X$ is a $\mathbb{Q}$-weakly equivariantly formal $G$-space.
\end{enumerate}
%If $\pi_1(G)$ is torsion-free, then the above are equivalent to 
%\begin{enumerate}
%	\item[(4)] $K_G^*(X; \mathbb{Q})$ is a free $R(G; \mathbb{Q})$-module.
%\end{enumerate}
\end{theorem}
We will also give alternative proofs of equivariant formality of certain group actions which were proved in cohomological terms. These are conjugation action on compact Lie groups, left translation action on generalized flag manifolds, and compact Lie group actions with maximal rank isotropy subgroups. 

We note that there is an analogue of Theorem \ref{rkef&rcef} in the algebro-geometric setting (\cite[Theorem 1.1]{Gr}): it is also an assertion of surjectivity, but of the forgetful map from the rational Grothendieck group of $G$-equivariant coherent sheaves on a $G$-scheme $X$ to the corresponding Grothendieck group for ordinary coherent sheaves, where $G$ is a connected reductive algebraic group. Theorem \ref{rkef&rcef} confirms the expectation (\cite[Introduction]{Gr}) that the $K$-theoretic forgetful map is onto for equivariantly formal topological spaces. 

In the remainder of this note, the coefficient ring of any cohomology theory is always $\mathbb{Q}$.

\emph{Acknowledgment.} We would like to gratefully acknowledge the anonymous referee for the critical comments on the early drafts of this paper and especially the suggestions for improving Section 3.3. We would like to thank Ian Agol for answering a question related to the proof of Theorem \ref{eqfmax}.

\section{The proof} From now on, unless otherwise specified, $X$ is a finite CW-complex equipped with an action by a torus $T$ or more generally a compact connected Lie group $G$. The following $K$-theoretic abelianization result enables us to prove $K$-theoretic results in this Section in the $T$-equivariant case first and then generalize to the $G$-equivariant case. 
%We use $K$-theory with complex coefficient, and denote such with $K^*(X)$ and $K_G^*(X)$ by abuse of notation. Along the same vein, $R(G)$ is used to denote the representation ring with complex coefficient, thus $\mathbb{C}$ is regarded as a $R(G)$-module through the augmentation homomorphism. We use the decoration $\mathbb{Z}$ if integral coefficient is used. For instance integral $K$-theory of $M$ is denoted by $K^*(M, \mathbb{Z})$. The term `rational weakly equivariantly formal' is used to refer to, in view of Definition \ref{haradaland}, the condition that
%\[K_G^*(M)\otimes_{R(G)}\mathbb{C}\to K^*(M)\]
%is a ring isomorphism. 

\begin{theorem}[cf. {\cite[Theorem 4.9(ii)]{HLS}}]\label{GtoT}	%\item\textnormal{({\cite[Proposition 4.9]{A}})}\label{atiyahinj}
		%The map $r^*: K_G^*(X; k)\to K_T^*(X; k)$ restricting the $G$-action to the $T$-action is a split injective map for any $k$.  
		%\item
Let $T$ be a maximal torus of $G$ and $W$ the Weyl group. The map $r^*: K_G^*(X; \mathbb{Q})\to K_T^*(X; \mathbb{Q})$ restricting the $G$-action to the $T$-action is an injective map onto $K^*_T(X; \mathbb{Q})^W$. Here if $w\in W$ and $V$ is an equivariant $T$-vector bundle, $w$ takes $V$ to the same underlying vector bundle with $T$-action twisted by $w$, and this $W$-action on the set of isomorphism classes of equivariant $T$-vector bundles induces the $W$-action on $K_T^*(X)$.
	%\end{enumerate}
\end{theorem}
%We will also use the following result on equivariant Chern character on several occasions.
\begin{definition}
	Let $H^{**}_G(X)$ be the completion of $H_G^*(X)$ as a $H_G^*(\text{pt})$-module at the augmentation ideal $J:=H_G^{+}(\text{pt})$ (cf. the paragraph preceeding \cite[Proposition 2.8]{RK}). 
\end{definition}	
The equivariant Chern character for a finite CW-complex with a $G$-action is the map
\[\text{ch}_G: K_G^*(X; \mathbb{Q})\to H_G^{**}(X)\]
which is defined by applying the Borel construction to the non-equivariant Chern character (cf. the discussion before \cite[Lemma 3.1]{RK}). Like the non-equivariant Chern character, $\text{ch}_G$ maps $K_G^0(X; \mathbb{Q})$ to the even degree part of $H_G^{**}(X)$ and $K_G^{-1}(X; \mathbb{Q})$ to the odd degree part. The image of $\text{ch}_G$ lies in $H_G^{**}(X)$ for the following reason which is borrowed from the proof of \cite[Lemma 3.1]{RK}: as $X$ is a finite CW-complex, we can choose $a_1, a_2, \cdots, a_m\in H_G^*(X)$ which generate $H_G^*(X)$ as a $H_G^*(\text{pt})$-module. Let 
\[a_i\cdot a_j=\sum_{k=1}^m f_{ij}^k a_k\]
for $f_{ij}^k\in H_G^*(\text{pt})$, and $c$ be $c_1^G(L)$ for some $G$-equivariant line bundle $L$ such that 
\[c=\sum_{i=1}^m g_ia_i\]
for $g_i\in H_G^*(\text{pt})$. So
\[\text{ch}_G(L)=e^c=1+\sum_i g_ia_i+\frac{1}{2}\sum_{i, j, k}g_ig_jf_{ij}^k a_k+\frac{1}{6}\sum_{i, j, k, l, p}g_ig_jg_lf_{ij}^kf_{kl}^pa_p+\cdots.\]
Write $\text{ch}_G(L)=1+\sum_{i=1}^mp_ia_i$, where $p_i$ are power series in $g_i$ and $f_{ij}^k$. Identifying $g_i$ and $f_{ij}^k$ with $W$-invariant polynomials on $\mathfrak{t}$ through the identification $H_G^*(\text{pt})\cong H_T^*(\text{pt})^W\cong S(\mathfrak{t}^*)^W$ and using the estimate for $p_i$ given in the proof of \cite[Lemma 3.1]{RK}, we have that $p_i$ are in $H_G^{**}(\text{pt})$ and hence $ch_G(L)\in H_G^{**}(X)$. The assertion $\text{ch}_G(E)\in H_G^{**}(X)$ for general equivariant $G$-vector bundle $E$ follows from the splitting principle.

\begin{proposition}\label{equivchern}
	Let $G$ be a compact connected Lie group acting on a finite CW-complex $X$. Then the equivariant Chern character 
	\[\text{ch}_G: K^*_G(X; \mathbb{Q})\to H_G^{**}(X)\]
	is injective, and $\text{ch}_G^{-1}(J)=I(G; \mathbb{Q})$ when $X$ is a point. 
\end{proposition}
\begin{proof}
	By \cite[Theorem 2.1]{AS}, $K^*(X\times_G EG)\cong K^*_G(X\times EG)$ is the completion of $K^*_G(X)$ at $I(G)$. The map $\iota: K^*_G(X)\to K^*(X\times_GEG)$ induced by the projection map $X\times EG\to X$ is injective because the $I(G)$-adic topology of the completion is Hausdorff if $G$ is connected (cf. the Note immediately preceding \cite[Section 4.5]{AH}). It follows that the rationalized map $\iota\otimes\mathbb{Q}: K^*_G(X; \mathbb{Q})\to K^*(X\times_G EG; \mathbb{Q})$ is injective as well. On the other hand, let $EG_n$ be the Milnor join of $n$ copies of $G$. Then $X\times_G EG_n$ is compact and the ordinary Chern character map $\text{ch}_n: K^*(X\times_G EG_n; \mathbb{Q})\to H^*(X\times_G EG_n)$ is an isomorphism. Note that 
	\[K^*(X\times_G EG; \mathbb{Q})\cong\lim_{\substack{\longleftarrow\\n}}K^*(X\times_G EG_n; \mathbb{Q})\]
	(see \cite[Corollary 2.4, Proposition 4.1 and proof of Proposition 4.2]{AS}). It follows that the map 
	\[\text{ch}: K^*(X\times_G EG; \mathbb{Q})\to H_G^{**}(X)\]
	is the inverse limit of the isomorphisms $ch_n$ and injective by the left-exactness of inverse limit. The map $\text{ch}_G$ is the composition of the two injective maps $\iota\otimes\mathbb{Q}$ %(here $K^*_G(X; \mathbb{Q})\string^$ means the completion of $K_G^*(X; \mathbb{Q})$ at $I(G; \mathbb{Q})$, and the isomorphism is asserted by \cite[Theorem 2.1]{AS})
	and $\text{ch}: K^*(X\times_G EG; \mathbb{Q})\to H_G^{**}(X)$. Therefore $\text{ch}_G$ is injective. Next, consider the commutative diagram
	\begin{eqnarray*}
		\xymatrix{R(G; \mathbb{Q})\ar[r]\ar[d]_{\text{ch}_G}&K^*(\text{pt}; \mathbb{Q})\ar[d]^{\text{ch}}\\ H_G^{**}(\text{pt})\ar[r]& H^*(\text{pt})}
	\end{eqnarray*}
	where the two horizontal maps are forgetful maps. Since $J$ is the kernel of the bottom map and both $\text{ch}_G$ and $\text{ch}$ are injective, $\text{ch}_G^{-1}(J)$ is the kernel of the top map, which is precisely $I(G; \mathbb{Q})$. 
\end{proof}

Under the condition of weak equivariant formality, \cite[Proposition 4.2]{HL} asserts that the kernel of $f$ is $I(G)\cdot K_G^*(X)$. In fact, we also have
%We find that under the conditions that $G$ be a torus group $T$, $M$ be compact, and using complex coefficient, the weak equivariant formality condition can be removed. 

\begin{lemma}\label{kerforgetm}
	%Let $ET^m$ be the join of $m$ copies of $T$. Viewing $K^*(X\times ET^m/T)$ as a module over $K^*(ET^m/T)=\mathbb{C}[t_1, \cdots, t_n]/((t_1-1)^m, \cdots, (t_n-1)^m)$, we have that the kernel of the forgetful map 
	%\[f_m: K^*(X\times ET^m/T)\to K^*(X)\]
	%is $I(T)\cdot K^*(X\times ET^m/T)$, where $I(T)=(t_1-1, \cdots, t_n-1)$.
	Let $X$ be a finite CW-complex which is acted on by a compact connected Lie group $G$ equivariantly formally. Then the kernel of the forgetful map
	\[f_G\otimes\text{Id}_\mathbb{Q}: K_G^*(X; \mathbb{Q})\to K^*(X; \mathbb{Q})\]
	is $I(G; \mathbb{Q})\cdot K_G^*(X; \mathbb{Q})$.
\end{lemma}
\begin{proof}
	%We first consider the $T$-equivariant case, where $T$ is a maximal torus of $G$. 
	In the following diagram, 
	\begin{eqnarray}\label{cherndiag}
		\xymatrix{K_G^*(X; \mathbb{Q})\ar[r]^{f_G\otimes\text{Id}_\mathbb{Q}}\ar[d]_{\text{ch}_G}&K^*(X; \mathbb{Q})\ar[d]^{\text{ch}}\\ H^{**}_G(X)\ar[r]^{\widetilde{g}_G\otimes\text{Id}_\mathbb{Q}}& H^*(X)}
	\end{eqnarray}
	%$\text{ch}_T$ is the $T$-equivariant Chern character and
	where $\widetilde{g}_G\otimes\text{Id}_\mathbb{Q}$ is the forgetful map, $H_G^{**}(X)$ is the completion of $H_G^*(X)$ at the augmentation ideal $J$ of $H_G^*(\text{pt})$. Since $X$ is an equivariantly formal $G$-space, $H_G^*(X)$ is isomorphic to $H_G^*(\text{pt})\otimes H^*(X)$ as a $H_G^*(\text{pt})$-module, and the forgetful map
	\[g_G\otimes\text{Id}_\mathbb{Q}: H_G^*(X)\to H^*(X)\]
	has $J\cdot H_G^*(X)$ as the kernel. Since $H_G^*(X)$ is a finitely generated module over the Noetherian ring $H_G^*(\text{pt})$, a simple result on completions (cf. \cite[Theorem 55]{Ma}) implies that $H^{**}_G(X)\cong H_G^*(X)\otimes_{H^*_G(\text{pt})}H^{**}_G(\text{pt})$. So the kernel of $\widetilde{g}_G\otimes\text{Id}_\mathbb{Q}$ is $J\cdot H^{**}_G(X)$. By Proposition \ref{equivchern}, the preimage $\text{ch}_G^{-1}(J)$ is $I(G; \mathbb{Q})$ and $\text{ch}_G$ is injective. It follows that the kernel of $f_G\otimes\text{Id}_\mathbb{Q}$ is $\text{ch}_G^{-1}(J\cdot H_G^{**}(X))=I(G; \mathbb{Q})\cdot K_G^*(X; \mathbb{Q})$. 
\end{proof}

%\begin{remark}
%	If the equivariant formality assumption is dropped from Lemma \ref{kerforgetm}, then the kernel of the forgetful map $f$ (resp. $g$) is not necessarily $I(T)\cdot K_T^*(X)$ (resp. $J(T)\cdot H_T^*(X)$). An example is provided by $X$ a compact Lie group $G$ acted on by a torus subgroup $S$ by left multiplication, where $G/S$ is not a formal manifold (cf. Definition \ref{formalSullivan} and Theorem \ref{cohomologyhomog}).
%\end{remark}

\begin{proof}[Proof of Theorem \ref{rkef&rcef}, $(1)\iff (2)$]
	We first deal with the $T$-equivariant case, where $T$ is a maximal torus of $G$. We claim that, if $X$ is an equivariantly formal $T$-space, we have the following string of (in)equalities.
	\[\text{dim}_\mathbb{Q} K^*(X^T; \mathbb{Q})=\text{rank}_{R(T; \mathbb{Q})}K_T^*(X; \mathbb{Q})\leq \text{dim } K_T^*(X; \mathbb{Q})/I(T; \mathbb{Q})\cdot K_T^*(X; \mathbb{Q})\leq \text{dim } K^*(X; \mathbb{Q}).\]
	Applying Segal's localization theorem to the case of torus group actions (cf. \cite[Proposition 4.1]{Se}), we have that the restriction map $K_T^*(X; \mathbb{Q})\to K_T^*(X^T; \mathbb{Q})$ becomes an isomorphism after localizing at the zero prime ideal, i.e. to the field of fraction of $R(T; \mathbb{Q})$. So $\text{rank}_{R(T; \mathbb{Q})}K_T^*(X; \mathbb{Q})=\text{rank}_{R(T; \mathbb{Q})}K_T^*(X^T; \mathbb{Q})$. By \cite[Proposition 2.2]{Se}, $K_T^*(X^T; \mathbb{Q})$ is isomorphic to $R(T; \mathbb{Q})\otimes K^*(X^T; \mathbb{Q})$, whose rank over $R(T; \mathbb{Q})$ equals $\text{dim}_\mathbb{Q} K^*(X^T; \mathbb{Q})$. The first equality then follows. Next, by \cite[Proposition 5.4]{Se} and the discussion thereafter, we have that $K^*_T(X; \mathbb{Q})$ is a finite $R(T; \mathbb{Q})$-module. After localizing $K_T^*(X; \mathbb{Q})$ at $I(T; \mathbb{Q})$ and reduction modulo the same ideal, we have that $K_T^*(X; \mathbb{Q})_{I(T; \mathbb{Q})}/I(T; \mathbb{Q})\cdot K_T^*(X; \mathbb{Q})_{I(T; \mathbb{Q})}$ is a finite dimensional $\mathbb{Q}$-vector space. We let $n$ be the dimension of this vector space, and $x_1, \cdots, x_n\in K^*_T(X; \mathbb{Q})_{I(T; \mathbb{Q})}/I(T; \mathbb{Q})\cdot K_T^*(X; \mathbb{Q})_{I(T; \mathbb{Q})}$ be its basis. Finite generation of $K_T^*(X; \mathbb{Q})$ as a module over the Noetherian ring $R(T; \mathbb{Q})$ enables us to invoke Nakayama lemma, and have that there exist lifts $\widehat{x}_1, \cdots, \widehat{x}_n\in K_T^*(X; \mathbb{Q})_{I(T; \mathbb{Q})}$ that generate $K_T^*(X; \mathbb{Q})_{I(T; \mathbb{Q})}$ as a $R(T; \mathbb{Q})_{I(T; \mathbb{Q})}$-module. It follows, after further localization to the field of fraction of $R(T; \mathbb{Q})$, that $\widehat{x}_1, \cdots, \widehat{x}_n$ span $K_T^*(X; \mathbb{Q})_{(0)}$ as a $R(T; \mathbb{Q})_{(0)}$-vector space, and that 
	\[\text{dim}_{R(T; \mathbb{Q})_{(0)}} K_T^*(X; \mathbb{Q})_{(0)}\leq \text{dim}_\mathbb{Q} K_T^*(X; \mathbb{Q})_{I(T; \mathbb{Q})}/I(T; \mathbb{Q})\cdot K_T^*(X; \mathbb{Q})_{I(T; \mathbb{Q})}=n.\] 
	Noting the isomorphism $K_T^*(X; \mathbb{Q})/I(T; \mathbb{Q})\cdot K_T^*(X; \mathbb{Q})\cong K_T^*(X; \mathbb{Q})_{I(T; \mathbb{Q})}/I(T; \mathbb{Q})\cdot K_T^*(X; \mathbb{Q})_{I(T; \mathbb{Q})}$, we arrive at the first inequality. Finally, the last inequality follows from Lemma \ref{kerforgetm}. 

	If $X$ is an equivariantly formal $T$-space, then $\text{dim}H^*(X)=\text{dim}H^*(X^T)$ (see \cite[p. 46]{Hs}). The Chern character isomorphism implies that $\text{dim}K^*(X^T; \mathbb{Q})=\text{dim}K^*(X; \mathbb{Q})$ which, together with the (in)equalities in the above claim, yields $\text{dim}K_T^*(X; \mathbb{Q})/I(T; \mathbb{Q})\cdot K_T^*(X; \mathbb{Q})=\text{dim}K^*(X; \mathbb{Q})$ or, equivalently, that $X$ is RKEF.
	
	Assume on the other hand that $X$ is RKEF. Consider the commutative diagram (\ref{cherndiag}). Since $f_T\otimes\text{Id}_\mathbb{Q}$ is onto and $\text{ch}$ is an isomorphism, $\widetilde{g}_T\otimes\text{Id}_\mathbb{Q}$ is onto. By \cite[Theorem 55]{Ma}, we have that $H_T^{**}(X)\cong H_T^*(X)\otimes_{H^*_T(\text{pt})}H_T^{**}(\text{pt})$. Applying $\widetilde{g}_T\otimes \text{Id}_\mathbb{Q}$ gives $H^*(X)=\text{Im}(\widetilde{g}_T\otimes\text{Id}_\mathbb{Q})=\text{Im}({g}_T\otimes \text{Id}_\mathbb{Q})\otimes_\mathbb{Q}\mathbb{Q}=\text{Im}(g_T\otimes\text{Id}_\mathbb{Q})$. Hence $X$ is $T$-equivariantly formal.

	With the equivalence of equivariant formality and RKEF for $T$-action we have just proved and the fact that, if $T$ is a maximal torus of $G$ which is compact and connected, $T$-equivariant formality is equivalent to $G$-equivariant formality (cf. \cite[Proposition 2.4]{GR}), it suffices to show that $f_T\otimes\text{Id}_\mathbb{Q}$ is onto if and only if $f_G\otimes\text{Id}_\mathbb{Q}$ is onto in order to establish the equivalence of equivariant formality and RKEF for $G$-action. One direction is easy: if $f_G\otimes\text{Id}_\mathbb{Q}$ is onto, so is $f_T\otimes\text{Id}_\mathbb{Q}$ because $f_G\otimes\text{Id}_\mathbb{Q}=(f_T\otimes\text{Id}_\mathbb{Q})\circ r^*$. Conversely, suppose that $f_T\otimes\text{Id}_\mathbb{Q}$ is onto. Then any $x\in K^*(X; \mathbb{Q})$ admits a lift $\widetilde{x}\in K_T^*(X; \mathbb{Q})$. Note that for any $w\in W$, $(f_T\otimes\text{Id}_\mathbb{Q})(w\cdot \widetilde{x})=x$. It follows that the average 
	\[\overline{x}:=\frac{1}{|W|}\sum_{w\in W}w\cdot\widetilde{x}\]
	is also a lift of $x$. Moreover, by Theorem \ref{GtoT}, $\overline{x}\in r^*K_G(X; \mathbb{Q})$. So $(r^*)^{-1}(\overline{x})\in K_G^*(X; \mathbb{Q})$ is a lift of $x$ and $f_G\otimes\text{Id}_\mathbb{Q}$ is onto as well. 
\end{proof}

\begin{proof}[Proof of Theorem \ref{rkef&rcef}, $(1)\iff (3)$]
	That $\mathbb{Q}$-weakly equivariant formality implies RKEF is immediate (cf. \cite[Definition 4.1]{HL}). On the other hand, if $X$ is a RKEF $G$-space, then by Theorem \ref{rkef&rcef}, $(1)\Longrightarrow (2)$, $X$ is an equivariantly formal $G$-space. The map
	\begin{align*}
		K_G^*(X; \mathbb{Q})\otimes_{R(G; \mathbb{Q})}\mathbb{Q}&\to K^*(X; \mathbb{Q})\\
		\alpha\otimes z&\mapsto f_G(\alpha)z
	\end{align*}
	is injective by Lemma \ref{kerforgetm} and surjective by RKEF. Hence $X$ is a $\mathbb{Q}$-weakly equivariantly formal $G$-space. This completes the proof. 
\end{proof}

\section{Some applications}

In this Section, we shall demonstrate the utility of Theorem \ref{rkef&rcef} by giving alternative proofs of some previous results.

\subsection{Conjugation action on compact Lie groups}\label{conjugation} Let $G$ be a compact connected Lie group with conjugation action by itself. It is well-known that this action is equivariantly formal. See, for example, \cite[Sect. 11.9, Item 6]{GS}) for a sketch of proof for the case $G=U(n)$, and \cite{J} for an explicit construction of equivariant extensions of the generators of $H^*(G)$. We will show equivariant formality of conjugation action by proving that $G$ is a RKEF $G$-space. By \cite[II, Theorem 2.1]{Ho}, 
\[K^*(G; \mathbb{Q})\cong\bigwedge\nolimits_\mathbb{Q}^*(R\otimes\mathbb{Q}),\]
where $R$ is the image of the map
\[\delta: R(G)\to K^{-1}(G)\]
which sends $\rho\in R(G)$ to the following complex of vector bundles\footnote{The map $\delta$, which was defined in \cite{BZ} and corrected in \cite{F}, is the same as the map $\beta$ defined in \cite{Ho}.}
\begin{align*}
	0\longrightarrow G\times\mathbb{R}\times V &\longrightarrow G\times\mathbb{R}\times V\longrightarrow 0\\
	(g, t, v)&\mapsto\begin{cases}(g, t, -t\rho(g)v),\ \ &\text{if }t\geq 0, \\ (g, t, v),\ \ &\text{if }t\leq 0.\end{cases}
\end{align*}
For any $\rho$, $\delta(\rho)$ admits an equivariant lift in $K_G^*(G)$ because $G\times\mathbb{R}\times V$ can be equipped with the $G$-action given by 
\[g_0\cdot(g, t, v)=(g_0gg_0^{-1}, t, \rho(g_0)v), \]
with respect to which the middle map of the above complex of vector bundles is $G$-equivariant. Thus $f_G\otimes\text{Id}_\mathbb{Q}: K_G^*(G; \mathbb{Q})\to K^*(G; \mathbb{Q})$ is onto, i.e., $G$ is a RKEF $G$-space.

\subsection{Left translation action on $G/K$ where $\text{rank }G=\text{rank }K$}\label{equalranks} Let $G$ be a compact connected Lie group and $K$ a connected Lie subgroup of the same rank. The left translation action on $G/K$ by $G$ is well-known to be equivariantly formal, which can be proved by noting that $G/K$ satisfies the sufficient condition for equivariant formality that its odd cohomology vanish (cf. \cite[Chapter XI, Theorem VII]{GHV}). Alternatively, by the rationalized version of \cite[Theorem 4.2]{Sn} and the remark following it, 
\[K^*(G/K; \mathbb{Q})\cong R(K; \mathbb{Q})\otimes_{R(G; \mathbb{Q})}\mathbb{Q}\cong R(K; \mathbb{Q})/r^*I(G; \mathbb{Q}),\]
where $r^*: R(G; \mathbb{Q})\to R(K; \mathbb{Q})$ is the restriction map. The forgetful map $f_G\otimes\text{Id}_\mathbb{Q}: K^*_G(G/K; \mathbb{Q})\cong R(K; \mathbb{Q})\to K^*(G/K; \mathbb{Q})$ is simply the projection map and hence surjective (in fact the forgetful map sends any representation $\rho\in R(K)$ to the $K$-theory class of the homogeneous vector bundle $G\times_K V_\rho$, where $V_\rho$ is the underlying complex vector space for $\rho$). Thus $G/K$ is a RKEF $G$-space, and equivalently an equivariantly formal $G$-space.
\begin{remark}\label{rkefgood}
	In the more general case where equality of ranks of $G$ and $K$ is not assumed, a representation theoretic characterization of equivariant formality of the left translation action of $K$ on $G/K$ is given by virtue of RKEF in \cite{CF}. 
\end{remark}

\subsection{Actions with connected maximal rank isotropy subgroups} In this section we will prove the following equivariant formality result. 
\begin{theorem}\label{eqfmax}
	Let $G$ be a compact connected Lie group and $X$ a finite $G$-CW complex. Suppose that the $G$-action on $X$ has maximal rank connected isotropy subgroups. Then $X$ is an equivariantly formal $G$-space.
\end{theorem}	
\begin{remark}
	In fact, Theorem \ref{eqfmax} follows from \cite[Corollary 3.5]{GR}, where connectedness of isotropy subgroups is not assumed. Though the space under consideration in \cite[Corollary 3.5]{GR} is the subset of a compact $G$-manifold consisting of those points with maximal rank isotropy subgroups, its proof does not make use of this assumption and can be easily adapted to the more general case of $G$-CW complexes. Indeed the proof hinges on the observation that for any compact space $X$ with maximal rank isotropy subgroups and a maximal torus $T$, the map $G\times_{N_G(T)}X^T\to X$ given by $[g, x]\mapsto gx$ is onto and that the fibers of the map are acyclic. This enables one to assert the isomorphism $H_G^*(X)\cong H_{N_G(T)}^*(X^T)$. The latter, by abelianization, is $H_T^*(X^T)^W$, which in turn by a commutative algebra result (\cite[Lemma 2.7]{GR}) is a free module over $H_T^*(\text{pt})^W\cong H_G^*(\text{pt})$. Hence $X$ is an equivariantly formal $G$-space.
\end{remark}
\begin{remark}
	If $G$ in addition satisfies the condition that $\pi_1(G)$ be torsion-free, then $K^*_G(X; \mathbb{Q})$ is a free $R(G; \mathbb{Q})$-module with rank $\text{dim}_\mathbb{Q}K^*(X^T; \mathbb{Q})$ (\cite[Theorem 1.1]{AG}). 
\end{remark}
We would like to give a different proof of this result by using Theorem \ref{rkef&rcef} and induction on the dimension of $X$. We shall point out that the group actions considered in Sections \ref{conjugation} and \ref{equalranks} are examples of group actions we discuss in this section. However, equivariant formality of left translation actions on generalized flag manifolds as in Section \ref{equalranks} is used in the proof. 
\begin{lemma}\label{equiviso}
	Let $G$ be a compact connected Lie group acting on a finite CW-complex $X$ equivariantly formally. Let $V_1$ and $V_2$ be vector bundles on $X$ which are isomorphic nonequivariantly. Then there exist positive integers $a$ and $b$ such that $V_1^{\oplus a}\oplus\underline{\mathbb{C}}^b$ and $V_2^{\oplus a}\oplus\underline{\mathbb{C}}^b$ can be made equivariant $G$-vector bundles which are isomorphic equivariantly.
\end{lemma}
\begin{proof}
	By Theorem \ref{rkef&rcef} and the discussion preceding it, there exists $p$ and $q$ such that $T_1:=V_1^{\oplus p}\oplus\underline{\mathbb{C}}^q$ and $T_2:=V_2^{\oplus p}\oplus\underline{\mathbb{\mathbb{C}}^q}$ admit equivariant structures. Let $\widetilde{T}_1$ and $\widetilde{T}_2$ denote the corresponding equivariant $G$-vector bundles. They then define the equivariant $K$-theory class $[\widetilde{T}_1]-[\widetilde{T}_2]\in K_G^*(X; \mathbb{Q})$ which lies in the kernel of the forgetful map $f_G\otimes\text{Id}_\mathbb{Q}$. By Lemma \ref{kerforgetm}, there exist a positive integer $m$, representations $\rho_1^i$ and $\rho_2^i$ of $G$ with the same dimension, and equivariant $G$-vector bundles $A_i$ and $B_i$ such that
	\[m([\widetilde{T}_1]-[\widetilde{T}_2])=\sum_i([\underline{\rho_1^i}]-[\underline{\rho_2^i}])\cdot ([A_i]-[B_i])\]
	Here, for $\rho\in R(G)$ with $V_\rho$ being the complex vector space underlying the representation, $\underline{\rho}$ means the vector bundle $X\times V_\rho$ with the diagonal $G$-action. By the definition of Grothendieck construction, there exists an equivariant $G$-vector bundle $C$ such that we have the following $G$-vector bundle isomorphism. 
	\[\widetilde{T}_1^{\oplus m}\oplus\bigoplus_i(\underline{\rho_2^i}\otimes A_i\oplus\underline{\rho_1^i}\otimes B_i)\oplus C\cong \widetilde{T}_2^{\oplus m}\oplus \bigoplus_i(\underline{\rho_1^i}\otimes A_i\oplus\underline{\rho_2^i}\otimes B_i)\oplus C.\]
	By \cite[Proposition 2.4]{Se}, there exists an equivariant $G$-vector bundle $D$ such that $\bigoplus_i(\underline{\rho_2^i}\otimes A_i\oplus \underline{\rho_1^i}\otimes B_i)\oplus C\oplus D\cong \underline{\rho_0}$ for some $\rho_0\in R(G)$. Taking the direct sum of both sides with $D$ and forgetting the equivariant structures, we have
	\[V_1^{\oplus pm}\oplus \underline{\mathbb{C}}^{qm+\text{dim}\rho_0}\cong V_2^{\oplus pm}\oplus\underline{\mathbb{C}}^{qm+\text{dim}\rho_0}.\]
	Taking $a=pm$ and $b=qm+\text{dim}\rho_0$ finishes the proof. 
	%\[\widetilde{T}_1^{\oplus m}\oplus \underline{\rho_0}\cong\widetilde{T}_2^{\oplus m}\oplus \bigoplus_i(\underline{\rho_1^i}\otimes A_i\oplus\underline{\rho_2^i}\otimes B_i)\oplus C\oplus D.\]
\end{proof}

\begin{proof}[Proof of Theorem \ref{eqfmax}]
Consider the $n$-skeleton $X_n$. It is obtained by gluing the equivariant cells $G/K_i\times\mathbb{D}^n$ for $1\leq i\leq k$ and $K_i$ compact, connected and of maximal rank, to the $(n-1)$-skeleton $X_{n-1}$ through some $G$-equivariant attaching maps. For convenience of exposition and without loss of generality we will consider the case of attaching one equivariant cell $G/K\times\mathbb{D}^n$. Let
\[f: G/K\times\partial\mathbb{D}^n\to X_{n-1}\]
be the equivariant attaching map and 
\[F: G/K\times\mathbb{D}^n\to X_n\] 
be the inclusion of the equivariant cell into $X_n$. We also let $V$ be any given vector bundle over $X_n$. To prove Proposition \ref{eqfmax}, it suffices, by Theorem \ref{rkef&rcef} and the discussion after Definition \ref{rkef}, to show that, for some $p$ and $q$, $V^{\oplus p}\oplus \underline{\mathbb{C}}^{q}$ admits an equivariant structure, assuming by induction hypothesis that $V_0:=V|_{X_{n-1}}$ satisfies the condition that $V_0^{\oplus p_0}\oplus\underline{\mathbb{C}}^{q_0}$ admits an equivariant structure for some $p_0$ and $q_0$. 

Note that $V$ can be obtained by gluing $V_0\to X_{n-1}$ and $W\to G/K_i\times\mathbb{D}^n$, where $W:=F^*V$, through the clutching maps, i.e. vector bundle homomorphism
\[h: W|_{G/K\times\partial\mathbb{D}^n}\to V_0\]
which covers the map $f$ and send fiber to fiber isomorphically. By the discussion in Section \ref{equalranks} and the contractibility of $\mathbb{D}^n$, there exist $r$ and $s$ such that $W^{\oplus r}\oplus\underline{\mathbb{C}}^{s}$ is isomorphic to a certain homogeneous vector bundle which is obviously $G$-equivariant. If we take $p=\text{LCM}(p_0, r)$ and $q=\text{max}\{q_0, s\}$ then both $V_0^{\oplus p}\oplus \underline{\mathbb{C}}^{q}$ and $W^{\oplus p}\oplus\underline{\mathbb{C}}^q$ admit equivariant structures. Consider the clutching map
\[j: W^{\oplus p}|_{G/K\times\partial \mathbb{D}^n}\oplus \underline{\mathbb{C}}^{q}\to V_0^{\oplus p}\oplus\underline{\mathbb{C}}^q\]
built from $h$ for the vector bundles $W^{\oplus p}\oplus \underline{\mathbb{C}}^q$ and $V_0^{\oplus p}\oplus \underline{\mathbb{C}}^q$. The vector bundle $V^{\oplus p}\oplus \underline{\mathbb{C}}^{q}$ admits an equivariant structure if $j$ is homotopy equivalent to another clutching map which is $G$-equivariant. Now we define the map %Note that $j$ is the composition of 
\[\alpha: W^{\oplus p}|_{G/K\times\partial \mathbb{D}^n}\oplus\underline{\mathbb{C}}^q\to f^*V_0^{\oplus p}\oplus\underline{\mathbb{C}}^q\]
such that $j$ is the composition of $\alpha$ and the natural map
\begin{align*}
	f^*V_0^{\oplus p}\oplus\underline{\mathbb{C}}^q\cong f^*(V_0^{\oplus p}\oplus\underline{\mathbb{C}}^q)&\to V_0^{\oplus p}\oplus\underline{\mathbb{C}}^q\\
	(x, v)&\mapsto v,\text{ where }f(x)=\pi(v), x\in G/K\times\partial\mathbb{D}^n, v\in V_0^{\oplus p}\oplus\underline{\mathbb{C}}^q.
\end{align*}
The latter map is obviously $G$-equivariant. If the map $\alpha$ is homotopy equivalent to a $G$-equivariant map (and hence so is the clutching map $j$), then $V^{\oplus p}\oplus\underline{\mathbb{C}}^q$, which is obtained by gluing $V_0^{\oplus p}\oplus \underline{\mathbb{C}}^q$ and $W^{\oplus p}\oplus \underline{\mathbb{C}}^q$ through the clutching map, admits the $G$-equivariant structure inherited from those of $V_0^{\oplus p}\oplus \underline{\mathbb{C}}^q$ and $W^{\oplus p}\oplus\underline{\mathbb{C}}^q$. In fact it suffices to show the following 
\begin{claim}
There exist some positive integers $l$ and $m$ such that the map 
\[\alpha^{\oplus m}\oplus\text{Id}_{\underline{\mathbb{C}}^l}: (W^{\oplus p}|_{G/K\times\partial\mathbb{D}^n}\oplus\underline{\mathbb{C}}^q)^{\oplus m}\oplus\underline{\mathbb{C}}^l\to (f^*V_0^{\oplus p}\oplus\underline{\mathbb{C}}^q)^{\oplus m}\oplus\underline{\mathbb{C}}^l\]
is homotopy equivalent to a $G$-equivariant map. 
\end{claim}
The claim will imply that $V^{\oplus pm}\oplus\underline{\mathbb{C}}^{qm+l}$ admits an equivariant structure by the above clutching argument. We may then replace $p$ and $q$ with $pm$ and $qm+l$ respectively.%i.e., the former map is homotopy equivalent to a $G$-equivariant map for $q$ sufficiently large. 

We shall prove the above claim. Note that $\alpha$ is a vector bundle isomorphism as it covers the identity map on $G/K\times\partial\mathbb{D}^n$ and send fiber to fiber isomorphically. Bearing in mind that $G/K$ is an equivariantly formal $G$-space (cf. Section \ref{equalranks}) and so is $\partial\mathbb{D}^n$ due to the trivial $G$-action, $G/K\times\partial\mathbb{D}^n$ is an equivariant formal $G$-space because it is a product of equivariant formal $G$-spaces. By Lemma \ref{equiviso}, there exist positive integers $a$ and $b$ and equivariant $G$-vector bundle isomorphism
\[\beta: (f^*V_0^{\oplus p}\oplus \underline{\mathbb{C}}^q)^{\oplus a}\oplus\underline{\mathbb{C}}^b\to (W^{\oplus p}|_{G/K\times\partial\mathbb{D}^n}\oplus\underline{\mathbb{C}}^q)^{\oplus a}\oplus\underline{\mathbb{C}}^b.\]
%Thus we may think of $\alpha$ as a vector bundle automorphism of $W^{\oplus p}|_{G/K\times\partial\mathbb{D}^n}\oplus\underline{\mathbb{C}}^q$, which for brevity will be denoted by $U$ from now on. 
The composition $\gamma:=\beta\circ (\alpha^{\oplus a}\oplus \text{Id}_{\underline{\mathbb{C}}^b})$ then is a vector bundle automorphism of $U:=W^{\oplus pa}|_{G/K\times\partial\mathbb{D}^n}\oplus\underline{\mathbb{C}}^{qa+b}$. Let $Y$ be the vector bundle $U\times[0, 1]/((u, 0)\sim(\gamma(u), 1))$ over $G/K\times\partial\mathbb{D}^n\times S^1$, which is an equivariantly formal $G$-space by the above argument. %This defines a $K$-theory class $[Y]\in K^0(G/K\times\partial\mathbb{D}^n\times S^1)$. 
%Bearing in mind that $G/K$ is an equivariantly formal $G$-space (cf. Section \ref{equalranks}) and so are $\partial\mathbb{D}^n$ and $S^1$ due to the trivial $G$-action, $G/K\times\partial\mathbb{D}^n\times S^1$ is an equivariantly formal $G$-space as it is a product of equivariantly formal $G$-spaces. 
By Theorem \ref{rkef&rcef}, $G/K\times\partial\mathbb{D}^n\times S^1$ is RKEF. It follows that for some positive integers $c$ and $d$,  $Y^{\oplus c}\oplus\underline{\mathbb{C}}^d$ can be made an equivariant $G$-vector bundle, and thus $\gamma^{\oplus c}\oplus \text{Id}_{\underline{\mathbb{C}}^d}$ is homotopy equivalent to some $G$-equivariant clutching map $\delta: U^{\oplus c}\oplus\text{Id}_{\underline{\mathbb{C}}^d}\to U^{\oplus c}\oplus\text{Id}_{\underline{\mathbb{C}}^d}$. It follows that $\alpha^{\oplus ac}\oplus\text{Id}_{\underline{\mathbb{C}}^{bc+d}}=((\beta)^{-1})^{\oplus c}\oplus\text{Id}_{\underline{\mathbb{C}}^d})\circ(\gamma^{\oplus c}\oplus\text{Id}_{\underline{\mathbb{C}}^d})$ is homotopy equivalent to the equivariant $G$-vector bundle isomorphism $((\beta^{-1})^{\oplus c}\oplus\text{Id}_{\underline{\mathbb{C}}^d})\circ\delta$. Now taking $m=ac$ and $l=bc+d$ finishes the proof of the claim. %$\alpha^{\oplus m}\oplus\text{Id}_{\underline{\mathbb{C}}^l}$ is homotopy equivalent to some $G$-equivariant clutching map. %We shall consider the following two cases depending on the parity of $n$. 

We have shown that, by induction on the dimension of $X$, for any given vector bundle $V\to X$, $V^{\oplus p}\oplus \underline{\mathbb{C}}^q$ admits an equivariant structure for some $p$ and $q$. The same is true for the suspension $\Sigma X$ because it is also a $G$-CW complex with maximal rank connected isotropy subgroups. It follows that the $G$-action on $X$ is equivariantly formal by Theorem \ref{rkef&rcef}. 
\end{proof}

\noindent\footnotesize{\textsc{National Center for Theoretical Sciences,
Mathematics Division, National Taiwan University, Taipei 10617, Taiwan\\
\\
School of Mathematical Sciences, the University of Adelaide, Adelaide, SA 5005, Australia}\\
\\
\textsc{E-mail}: \texttt{chi-kwong.fok@adelaide.edu.au}\\
\textsc{URL}: \texttt{https://sites.google.com/site/alexckfok}

\end{document}